\documentclass{amsart}
\usepackage{amsmath,amssymb,amsthm,bm}
\bmdefine{\be}{e}
\bmdefine{\bg}{g}
\bmdefine{\bk}{k}
\bmdefine{\bm}{m}
\bmdefine{\bp}{p}
\bmdefine{\bs}{s}
\bmdefine{\bt}{t}
\bmdefine{\bw}{w}

\newtheorem{thm}{THEOREM}[section]

\newtheorem{cor}{COROLLARY}
\newtheorem{prop}[thm]{PROPOSITION}
\newtheorem{rem}{REMARK}
\newtheorem{exa}{Example}[section]
\newcommand{\Romannum}[1]{\uppercase\expandafter{\romannumeral#1}}
\makeatletter
  
  \@addtoreset{equation}{section}
\makeatother
\begin{document}
\title{Hensel's lemma for general continuous functions}
\author{Hajime Kaneko}
\address{Institute of Mathematics, University of Tsukuba, 1-1-1, Tennodai, Tsukuba, Ibaraki, 305-8571, JAPAN; 
Research Core for Mathematical Sciences, University of Tsukuba, 1-1-1, Tennodai, Tsukuba, Ibaraki, 305-8571, JAPAN}
\email{kanekoha@math.tsukuba.ac.jp}
\author{Thomas Stoll}
\address{1. Universit\'e de Lorraine, Institut Elie Cartan de Lorraine, UMR 7502, Vandoeuvre-l\`es-Nancy, F-54506, France;
2. CNRS, Institut Elie Cartan de Lorraine, UMR 7502, Vandoeuvre-l\`es-Nancy, F-54506, France}
\email{thomas.stoll@univ-lorraine.fr}
\begin{abstract}

In the present paper, we generalize the well-known Hensel's lifting lemma to any continuous function $f : \mathbb{Z}_p\rightarrow \mathbb{Z}_p$. 
This answers a question posed by Axelsson and Khrennikov (2016) who showed the validity of Hensel's lemma for $1$- and for $p^\alpha$-Lipschitz functions.
For the statement and the proof, we introduce a suitable generalization of the original van der Put series. We use the concept of approximability of continuous functions to give numerical examples.

\end{abstract}
 \maketitle
\section{Introduction}
Denote by \(\mathbb{N}\) (resp. \(\mathbb{Z}^{+}\)) the set of nonnegative integers (resp. positive integers). Let \(p\) be a prime number and \(\mathbb{Z}_p\) the ring of \(p\)-adic integers. 
We denote \(v_p(u)\) (resp. \(|u|_p\)) the \(p\)-adic order (resp. the \(p\)-adic absolute value) of 
\(u\in\mathbb{Z}_p\) normalized with \(v_p(p)=1\) (resp. \(|p|_p=p^{-1}\)). 
We first recall Hensel's lifting lemma. Let \(Q(X)\in\mathbb{Z}_p[X]\backslash\mathbb{Z}_p\) 
and \(u\in\mathbb{Z}_p\). Assume that 
\[v_p(Q(u))>2v_p(Q'(u)).\]
Then there exists a unique \(\xi\in\mathbb{Z}_p\) satisfying 
\[v_p(\xi-u)>v_p(Q'(u))\]
and
\[Q(u)=0.\]
Axelsson and Khrennikov \cite{axe} generalized Hensel's lemma to \(1\)-Lipschitz functions, namely to functions
\(f:\mathbb{Z}_p\to\mathbb{Z}_p\) which satisfy
\[|f(x)-f(y)|_p\leq |x-y|_p\]
for any \(x,y\in\mathbb{Z}_p\) (from this, they derived also a generalization for \(p^\alpha\)-Lipschitz functions, $\alpha\geq 1$). For the statement and the proof of their generalization, they used van der Put series applied to this class of functions. Van der Put series date back to the work of M. van der Put \cite{vdP1, vdP2} and have shown to be very fruitful in $p$-adic analysis, dynamical systems and ergodic theory (see \cite{Sch}) as well as in various applications in cryptography, computer science and numerical analysis (see \cite{AKbook} for a detailed account on the various applications). For the application to dynamical systems it is natural to investigate the minimal conditions on the functions that allow such series expansions.\par

Axelsson and Khrennikov \cite[p.224, Remark 2.5]{axe} asked for a generalization of Hensel's lemma to any continuous function \(f:\mathbb{Z}_p\to\mathbb{Z}_p\) which would lead to an extension to their result. In this article, we provide such a generalization of Hensel's lemma to any continuous function, introducing the notion of generalized 
van der Put series and approximability of continuous functions. 
We consider the existence of the zeros of continuous functions and discuss uniqueness under certain conditions. \par
We first recall Theorem 2.4 of \cite{axe} (notation and statements are slightly modified for convenience to generalization; cf. \cite{ana}). 
For a positive integer \(m\), we denote its base-\(p\) expansion 
by 
\begin{align}
m=m_0+m_1p+\cdots+m_{k}p^k,
\label{eqn:1-1}
\end{align}
where \(k=k(m)=\lfloor \log_p m\rfloor\) and \(\log_p m=\log m/\log p\). 
In other words, \(k\) is a unique positive integer with 
\begin{align}
p^k\leq m<p^{1+k},
\label{eqn:1-2}
\end{align}
and \(m_0,m_1,\ldots,m_k\in\{0,1,\ldots,p-1\}\), \(m_k\ne 0\). For convenience, if \(m=0\), put 
\begin{align}
k=k(0)=\lfloor \log_p 0\rfloor:=0.
\label{eqn:1-3}
\end{align} 
Set 
\[
\chi(m,x)=
\begin{cases}
1 & \mbox{if }|x-m|_p\leq p^{-1-k},\\
0 & \mbox{ otherwise}
\end{cases}
\]
for any \(x\in\mathbb{Z}_p\). 
Let \(f:\mathbb{Z}_p\to\mathbb{Z}_p\) be any continuous function. 
Then there exists a unique sequence \(B(f;m)\) (\(m=0,1,\ldots\)) such that 
\begin{align}
f(x)=\sum_{m=0}^{\infty} B(f;m) \chi(m;x)
\label{eqn:1-4}
\end{align}
for any \(x\in\mathbb{Z}_p\), where (\ref{eqn:1-4}) is called the van der Put series of \(f\). 
Moreover, \(B(f;m)\) satisfies 
\[
B(f;m)=\begin{cases}
f(m) & \mbox{if }m< p, \\
f(m)-f(m-m_k p^k) & \mbox{otherwise}.
\end{cases}
\]
Anashin, Khrennikov and Yurova \cite[Theorem 5]{ana} showed that \(f\) is \(1\)-Lipschitz if and only if 
\begin{align}
v_p(B(f;m))\geq \lfloor \log_p m\rfloor
\label{eqn:1-5}
\end{align}
for any nonnegative integers \(m\). 
Thus, we see if \(f\) is \(1\)-Lipschitz, then 
\[b(f;m):= p^{-\lfloor \log_p m\rfloor} B(f;m)\in\mathbb{Z}_p.\]

We are now ready to state the main result of {\cite{axe}}. (The condition $m\geq R$, appearing in the second assumption in {\cite[Theorem 2.4]{axe}}, is a typo and has to be omitted.)
\begin{thm}[Theorem 2.4 in {\cite{axe}}]
Let \(f:\mathbb{Z}_p\to\mathbb{Z}_p\) be a \(1\)-Lipschitz function. Let \(R_0\) be a positive integer and 
\(u\) an integer with \(0\leq u< p^{R_0}\) and 
\[f(u)\equiv 0 \pmod {p^{R_0}}.\]
For any nonnegative integers \(m,R\) with \[m<p^R, \quad m\equiv u \pmod {p^{R_0}}, \quad R\geq R_0,\] assume that 
\begin{align*}
&\{b(f;m+i\cdot p^{R})\hspace{-2.5mm}\mod p\mid i=1,2,\ldots,p-1\}\\
&\hspace{35mm}=\{j \hspace{-2.5mm}\mod p\mid j=1,2,\ldots,p-1\}.
\end{align*}
Then there exists a unique \(\xi\in\mathbb{Z}_p\) satisfying \(f(\xi)=0\) and \(\xi\equiv u\pmod {p^{R_0}}\). 
\label{thm1}
\end{thm}

The fact that \(b(f;m)\in\mathbb{Z}_p\), which does not generally hold for continuous functions, is essential for the proof of Theorem \ref{thm1}. Thus, it is necessary to  
generalize van der Put series.

The structure of the paper is as follows. In Section \ref{sec2} we define generalized van der Put series by introducing a strictly increasing integer-valued function $\Phi$. This function corresponds to the concept of the modulus of continuity in the case of real-valued continuous functions. We then use this generalization in Section \ref{sec3} to classify general continuous functions and to state and prove our main result (Theorem \ref{thm2}). In Section \ref{sec4} we generalize the notion of derivative and approximability to this context. Along our way, we illustrate our results by some concrete examples. 

\section{Generalized van der Put series}\label{sec2}
This section is devoted to our generalization of van der Put series. Let \(\Phi:\mathbb{N}\to\mathbb{N}\) be a strictly increasing function. In the case of $\Phi=Id$, the identity function on \(\mathbb{N}\), we recover the original van der Put series.
To begin with, for any given \(p\)-adic integer 
\begin{align}\label{aaa}
x=\sum_{i=0}^{\infty} x_i p^i
\end{align} with \(x_i\in \{0,1,\ldots,p-1\}\) we split its terms according to $\Phi$, namely,
\begin{align}\label{bbb}
x=\sum_{j=0}^{\infty} \sum_{i=\Phi(-1+j)+1}^{\Phi(j)} x_i p^i,
\end{align}
where we denote \(\Phi(-1):=-1\) for convenience. For any nonnegative integer \(m\) put 
\[
\tau(m)=\tau(\Phi;m):=\min\left\{h\in\mathbb{N}\ \left| \ m<p^{1+\Phi(h)}\right.\right\}
\]
and, for any \(x\in\mathbb{Z}_p\), 
\[
\chi(\Phi,m;x):=
\begin{cases}
1 & \mbox{if }|x-m|_p\leq p^{-1-\Phi(\tau(m))},\\
0 & \mbox{otherwise}.
\end{cases}
\]
Then we see 
\(\tau(Id;m)=k\) by (\ref{eqn:1-2}), (\ref{eqn:1-3}), and so \(\chi(Id,m;x)=\chi(m;x)\) for any \(x\in\mathbb{Z}_p\) 
by \(\Phi(\tau(m))=k\). 
Let \(m\) be a nonnegative integer with \(\tau(m)\geq 1\), namely, \(m\geq p^{1+\Phi(0)}\). 
Denote the base-\(p\) expansion of \(m\) by (\ref{eqn:1-1}). 
Set 
\[
M(m)=M(\Phi;m):=\sum_{i=\Phi(-1+\tau(m))+1}^{\Phi(\tau(m))}m_i p^i.
\]
Note that 
\(M(Id;m)=m_k p^k\) by \(\Phi(\tau(m))=\Phi(-1+\tau(m))+1=k\). 
\begin{prop}
For any continuous function \(f:\mathbb{Z}_p\to\mathbb{Z}_p\) there exists a unique sequence 
\(B(m)=B(\Phi,f;m)\) (\(m=0,1,\ldots\)) such that 
\begin{align}
f(x)=
\sum_{m=0}^{\infty}B(\Phi,f;m)\chi(\Phi,m;x)
\label{eqn:2-1}
\end{align}
for any \(x\in\mathbb{Z}_p\). Moreover, \(B(\Phi,f;m)\) satisfies 
\begin{align}
B(\Phi,f;m)=
\begin{cases}
f(m) & \mbox{if }m<p^{1+\Phi(0)},\\
f(m)-f(m-M(m)) & \mbox{otherwise}.
\end{cases}
\label{eqn:2-2}
\end{align}
\label{prop1}
\end{prop}
We note that 
\(B(Id,f;m)=B(f;m)\) for any nonnegative integer \(m\). 
\begin{proof}[Proof of Proposition \ref{prop1}]
We first show that the sequence \(B(\Phi,f;m)\) (\(m=0,1,\ldots\)) defined inductively by (\ref{eqn:2-2}) 
satisfies (\ref{eqn:2-1}) (existence). For \(x\in\mathbb{Z}_p\) set
\begin{align}
F(x):=\sum_{m=0}^{\infty}B(\Phi,f;m)\chi(\Phi,m;x).
\label{eqn:2-3}
\end{align}
The convergence of the right-hand side of (\ref{eqn:2-3}) follows from 
the uniform continuity of \(f\) and 
\[\lim_{m\to\infty}|M(m)|_p=0\]
by the assumption that $\Phi$ is strictly increasing.
For any \(x\in\mathbb{Z}_p\) with (\ref{aaa}) and any nonnegative integer \(j\), 
set 
\begin{align}
x(j):=\sum_{i=0}^{\Phi(j)}x_i p^i.
\label{eqn:2-4}
\end{align}
Observe that 
\[F(x)=\sum_{j=0}^{\infty} \sum_{m\in\mathbb{N}, \tau(m)=j}B(\Phi,f;m)\chi(\Phi,m;x).\]
Let again \(j\) be a nonnegative integer. Then, for any \(m\in\mathbb{N}\) with \(\tau(m)=j\), we have that 
\begin{align}
\chi(\Phi,m;x)=1 \mbox{ if and only if } \tau(x(j))=j \mbox{ and }m=x(j).
\label{eqn:2-5}
\end{align}
It is easily seen that \(\tau(x(j))=j\) if and only if \(j=0\) or \(x(j)>x(j-1)\) because 
\[x(j)=x(j-1)+\sum_{i=\Phi(j-1)+1}^{\Phi(j)} x_i p^i\]
for \(j>0\). 
Letting 
\[J(x):=\{0\}\cup\{j\in\mathbb{Z}^{+}\mid x(j)>x(j-1)\},\]
we get by (\ref{eqn:2-2}) that, for any \(j\in J(x)\backslash \{0\}\), 
\[
B(\Phi,f;x(j))=f(x(j))-f(x(j-1))
\]
and that 
\begin{align*}
F(x)&=\sum_{j\in J(x)} B(\Phi,f;x(j))\\
&=f(x(0))+\sum_{j\in J(x)\backslash\{0\}} \big(f(x(j))-f(x(j-1))\big).
\end{align*}
Moreover, if \(j\in\mathbb{N}\backslash J(x)\), then we have \(f(x(j))-f(x(j-1))=0\). Thus, we see 
\[F(x)=f(x(0))+\sum_{j=1}^{\infty} \big(f(x(j))-f(x(j-1))\big)=f(x)\]
by \[\lim_{j\to\infty} x(j)=x.\]
Next, we prove that if  \(B(m)\) (\(m=0,1,\ldots\)) is a sequence satisfying 
(\ref{eqn:2-1}) for any \(x\in\mathbb{Z}_p\), then (\ref{eqn:2-2}) holds by 
induction on \(\tau(m)\) (uniqueness). We use the notation \(x(j) (j=0,1,\ldots)\) and \(J(x)\) in 
the case where \(x=m\). \par
First, we consider the case of \(\tau(m)=0\), namely, \(m<p^{1+\Phi(0)}\). 
Using (\ref{eqn:2-1}) and (\ref{eqn:2-5}), we get 
\[f(m)=B(m(0))=B(m)\]
by \(\tau(m(j))=0\) for any \(j\in\mathbb{N}\), 
which implies (\ref{eqn:2-2}). \par
We now assume that (\ref{eqn:2-2}) holds for any nonnegative integer \(m\) with \(\tau(m)\leq j\) for a fixed 
integer \(j\geq 0\). Let \(m\) be a nonnegative integer with \(\tau(m)=j+1\). Then we see 
\(J(m)\subset [0,j+1]\).  Using again (\ref{eqn:2-1}) and (\ref{eqn:2-5}), we get 
\begin{align}
f(m)=\sum_{l\in J(m)}B(m(l)).
\label{eqn:2-6}
\end{align}
Put 
\[J(m):=\{0=j_0<\cdots<j_{\kappa-1}<j_{\kappa}\},\]
where \(\kappa\geq 1\) and \(m(j_{\kappa})=m\) and \(m(j_{\kappa-1})=m-M(m)\). 
Hence, we obtain by (\ref{eqn:2-6}) and the inductive hypothesis that 
\begin{align*}
f(m)&=B(m(0))+\left(\sum_{l\in J(m) , 1\leq l\leq j_{\kappa-1}}B(m(l))\right)+B(m(j_{\kappa}))\\
&=f(m(0))+\left(\sum_{l\in J(m) , 1\leq l\leq j_{\kappa-1}}\big(f(m(l))-f(m(l-1))\big)\right)+B(m)\\
&=f(m(0))+\left(\sum_{1\leq l\leq j_{\kappa-1}}\big(f(m(l))-f(m(l-1))\big)\right)+B(m)\\
&=f(m(j_{\kappa-1}))+B(m)=f(m-M(m))+B(m),
\end{align*}
which implies (\ref{eqn:2-2}). Therefore, we proved the proposition. 
\end{proof}

\section{Hensel's lemma for continuous functions}\label{sec3}
In what follows, we use notation (\ref{aaa}) and (\ref{bbb}) for a \(p\)-adic integer \(x\). 
Let \(\Phi:\mathbb{N}\to\mathbb{N}\) be a strictly increasing function as used in the previous section. In order to state Hensel's lemma 
for continuous functions, we define a certain class of continuous functions based on the use of \(\Phi\). \par 
Let \(f:\mathbb{Z}_p\to\mathbb{Z}_p\) be a continuous function and \(n\) a positive integer. 
Since \(f\) is uniformly continuous on \(\mathbb{Z}_p\), there exists a nonnegative integer \(l\) such that, 
for any \(x,y\in\mathbb{Z}_p\), if \(|x-y|_p\leq p^{-l}\) then \(|f(x)-f(y)|_p\leq p^{-n}\). We denote the minimal value 
of such an \(l\) by \(\psi(f;n)\). Note that \(\psi(f;n)\) (\(n=1,2,\ldots\)) is non-decreasing. \par
Let \(\mathcal{F}(\Phi)\) be the class of continuous functions \(f:\mathbb{Z}_p\to\mathbb{Z}_p\) satisfying 
\[\psi(f;n)\leq 1+\Phi(n-1)\]
for any positive integer \(n\), namely, for any \(x,y\in\mathbb{Z}_p\), 
\begin{align}\label{xxx}
\mbox{if }
|x-y|_p\leq p^{-1-\Phi(n-1)}\mbox{, then }|f(x)-f(y)|_p\leq p^{-n}.
\end{align} 
Note that $|x-y|_p\leq p^{-1-\Phi(n-1)}$ is equivalent to 
\[\sum_{i=0}^{\Phi(n-1)}x_i p^i=\sum_{i=0}^{\Phi(n-1)}y_i p^i.\]
We see for any continuous \(f\) that 
\(f\in \mathcal{F}(Id)\) if and only if \(f\) is \(1\)-Lipschitz. \par
For any continuous function \(f\), we see that 
there exists a strictly increasing function \(\Phi\) such that \(f\in \mathcal{F}(\Phi)\), 
defining \(\Phi\) as follows: 
\begin{align*}
\Phi(0)&:=\max\{0,-1+\psi(f;1)\},\\
\Phi(n)&:=\max\{1+\Phi(n-1),-1+\psi(f;n+1)\}
\end{align*}
for any positive integer \(n\). \par
We now generalize (\ref{eqn:1-5}) as follows: 
\begin{prop}
Let \(\Phi:\mathbb{N}\to\mathbb{N}\) be a strictly increasing function and 
\(f:\mathbb{Z}_p\to\mathbb{Z}_p\) a continuous function. \\
Then \(f\in \mathcal{F}(\Phi)\) if and only if, for any nonnegative integer \(m\), 
\begin{align}
v_p(B(\Phi,f;m))\geq \tau(\Phi;m).
\label{eqn:3-1}
\end{align}
\label{prop2}
\end{prop}
\begin{rem}
\begin{rm}
In what follows, we consider the case where \(\Phi\) and \(f\) are fixed. 
For simplification, we also use the notation \(B(m)\), \(\tau(m)\), \(\chi(m;x)\), \(M(m)\) instead of 
\(B(\Phi,f;m)\), \(\tau(\Phi;m)\), \(\chi(\Phi,m;x)\), \(M(\Phi;m)\).
\end{rm}
\end{rem}
\begin{proof}[Proof of Proposition \ref{prop2}]
First, assume that \(f\in\mathcal{F}(\Phi)\). Let \(m\) be a nonnegative integer. 
For the proof of (\ref{eqn:3-1}), we may assume that \(\tau(m)\geq 1\). Observe by the 
definition of \(M(m)\) that 
\[v_p\big(m-(m-M(m))\big)=v_p(M(m))\geq 1+\Phi(-1+\tau(m)).\]
Applying (\ref{xxx}) with \(x=m, y=m-M(m)\) and \(n=\tau(m)\), we see 
\[v_p(B(m))\geq \tau(m)\]
by (\ref{eqn:2-2}), which implies (\ref{eqn:3-1}). \par
Conversely, we assume (\ref{eqn:3-1}). Let \(n\) be a positive integer and let \(x,y\in\mathbb{Z}_p\) with 
\(v_p(x-y)\geq 1+\Phi(n-1)\), namely, 
\[\sum_{i=0}^{\Phi(n-1)}x_i p^i=\sum_{i=0}^{\Phi(n-1)}y_i p^i.\]
Observe for any integer \(m\) with \(0\leq m <p^{1+\Phi(n-1)}\) that 
\[\chi(m;x)=\chi(m;y)\]
because \(\chi(m;x)\) (resp. \(\chi(m;y)\)) depends only on the first \(\Phi(n-1)\)-th digits of 
\(x\) (resp. \(y\)) by \(\tau(m)\leq n-1\). Thus, Proposition \ref{prop1} implies that 
\[f(x)-f(y)=
\sum_{m\geq p^{1+\Phi(n-1)}} B(m)\big(\chi(m;x)-\chi(m;y)\big).\]
For any integer \(m\) with \(m\geq p^{1+\Phi(n-1)}\), we get by (\ref{eqn:3-1}) that 
\[v_p(B(m))\geq \tau(m)\geq n,\]
and that 
\[v_p(f(x)-f(y))\geq n,\]
which implies that \(f\in \mathcal{F}(\Phi)\) because \(n\) is an arbitrary positive integer. 
\end{proof}
In what follows, for any \(f\in\mathcal{F}(\Phi)\) and \(m\in\mathbb{N}\), put 
\begin{align}
b(m)=b(\Phi,f;m):=p^{-\tau(\Phi;m)}B(\Phi,f;m)\in\mathbb{Z}_p
\label{eqn:3-2}
\end{align}
by Proposition \ref{prop2}. Moreover, for any \(p\)-adic integer \(x\) 
and nonnegative integer \(j\), put 
\[\rho(x;j):=p^{-1-\Phi(j-1)}\sum_{i=1+\Phi(j-1)}^{\Phi(j)}x_i p^i.\]
We now introduce a generalization of Theorem \ref{thm1} for continuous functions. 
\begin{thm}
Let \(\Phi:\mathbb{N}\to\mathbb{N}\) be a strictly increasing function and let \(f\in \mathcal{F}(\Phi)\). 
Let \(h,n_0\) and \(u\) be nonnegative integers with \(u<p^{1+\Phi(n_0)}\) and 
\begin{align}
  f(u)\equiv 0 \pmod{p^{1+h+n_0}}.\label{umod}
\end{align}
Suppose that there exists a sequence \((S(n))_{n\geq n_0}\) of sets of nonnegative integers with 
\[S(n)\subset \left(0,p^{\Phi(n+1)-\Phi(n)}\right), \quad \mbox{Card }S(n)=p-1,\]
where Card denotes the cardinality,
satisfying the following: For any nonnegative integers \(n,m\) with 
\[n\geq n_0, \quad m<p^{1+\Phi(n)}, \quad m\equiv u \pmod {p^{1+\Phi(n_0)}},\]
we have 
\begin{align}
&\left\{\left.b\left(\Phi,f;m+ip^{1+\Phi(n)}\right)\hspace{-2.5mm}\mod p^{h+1} \ \right| \  
i\in S(n)\right\}\nonumber
\\
&\hspace{40mm}
=\left\{\left. i p^h \mod p^{h+1}\right| i=1,2,\ldots,p-1\right\}.
\label{eqn:3-3}
\end{align}
Then there exists \(\xi\in\mathbb{Z}_p\) such that 
\begin{align}
&f(\xi)=0,\label{eqn:3-4}\\
&\xi\equiv u\pmod {p^{1+\Phi(n_0)}},\label{eqn:3-5}
\end{align}
and such that, for any \(n\geq n_0\), 
\begin{align}
\rho(\xi;n+1)\in \{0\}\cup S(n).\label{eqn:3-6}
\end{align}
Assume further that, for any integer \(m\) with \(m\geq p^{1+\Phi(n_0)}\) and 
\(m\equiv u \pmod {p^{1+\Phi(n_0)}}\), we have 
\begin{align}\label{yyy}
b(\Phi,f;m)\equiv 0 {\pmod {p^h}}.
\end{align}
Then \(\xi\in\mathbb{Z}_p\) satisfying (\ref{eqn:3-4}), (\ref{eqn:3-5}) and (\ref{eqn:3-6}) 
is unique. 
\label{thm2}
\end{thm}
\begin{rem}
\begin{rm}
As we will see in Section \ref{sec4}, if $f$ is differentiable modulo $p^s$ at $u$, then $h$ corresponds to the $p$-adic order of the derivative of $f$. If $\Phi=Id$ and $h=0$ then $1+n_0=1+\Phi(n_0)$ which is the value of $R_0$ in Theorem \ref{thm1}. Thus, Theorem \ref{thm2} with \(\Phi=Id\) and \(h=0\) implies Theorem \ref{thm1}.
\end{rm}
\end{rem}
\begin{rem}
\begin{rm}
Note that (\ref{yyy}) always holds if \(h=0\). 
\end{rm}
\end{rem}
\begin{rem}
\begin{rm}
A sequence \((S(n))_{n\geq n_0}\) satisfying the assumptions of Theorem \ref{thm2} is not generally unique. 
In Example \ref{exa1}, we consider the case where (\ref{yyy}) holds. We give two solutions \(\xi\) satisfying 
(\ref{eqn:3-4}) and (\ref{eqn:3-5}), by changing \((S(n))_{n\geq n_0}\). It would be an interesting question to try to prove the uniqueness of $\xi$ without assuming (3.9).
\end{rm}
\end{rem}
\begin{proof}[Proof of Theorem \ref{thm2}]
First, we prove the existence of such $\xi$ by using Newton's method. We show by induction on \(j\) that there exists a unique 
sequence \((u_j)_{j\geq n_0}\) of integers with 
\begin{align}
u_{n_0}=u\in [0,p^{1+\Phi(n_0)})\cap \mathbb{Z}\label{eqn:3-7}
\end{align}
satisfying the following: 
\begin{enumerate}
\item[(1)] For all $j\geq n_0$, 
\begin{align}
f(u_j)\equiv 0{\pmod {p^{1+h+j}}}.
\label{eqn:3-8}
\end{align}
\item[(2)] If \(j\geq 1+n_0\), then there exists \(i\in\{0\}\cup S(j-1)\) such that 
\begin{align}
u_j=u_{j-1}+i p^{1+\Phi(j-1)}, \quad 0\leq u_j<p^{1+\Phi(j)}.
\label{eqn:3-9}
\end{align}
\end{enumerate}
The case of \(j=n_0\) follows from (\ref{umod}). 
We assume that there exist unique integers \(u_{n_0},u_{1+n_0},\ldots,u_l\) for \(l\geq n_0\) satisfying 
(\ref{eqn:3-7}), (\ref{eqn:3-8}) and (\ref{eqn:3-9}) for \(j=n_0,1+n_0,\ldots,l\). Let \(i\in S(l)\). 
We see by \(p^{1+\Phi(l)}\leq i p^{1+\Phi(l)}\leq -p^{1+\Phi(l)}+p^{1+\Phi(l+1)}\) that 
\[\tau\left(u_l+i p^{1+\Phi(l)}\right)=l+1\]
and that 
\[M\left(u_l+i p^{1+\Phi(l)}\right)=i p^{1+\Phi(l)}.\]
Thus, using (\ref{eqn:3-2}) and (\ref{eqn:2-2}) with \(m=u_l+i p^{1+\Phi(l)}\), we get 
\begin{align}
f\left(
u_l+i p^{1+\Phi(l)}
\right)
=f(u_l)+p^{l+1}b\left(
u_l+i p^{1+\Phi(l)}
\right).
\label{eqn:3-10}
\end{align}
Since \(u_l \equiv u_{n_0}=u {\pmod {p^{1+\Phi(n_0)}}}\) by (\ref{eqn:3-9}) with \(j=n_0,1+n_0,\ldots,l\), 
we can apply (\ref{eqn:3-3}) with \(n=l,m=u_l\). 
Since \(f(u_l)\equiv 0\mod p^{l+1+h}\), selecting suitable \(i\in \{0\}\cup S(l)\), we obtain that 
\[u_{l+1}:=u_l+i p^{1+\Phi(l)}(<p^{1+\Phi(l+1)})\]
is a unique integer satisfying (\ref{eqn:3-8}) and (\ref{eqn:3-9}) with \(j=l+1\).
Putting \(\xi:=\lim_{j\to\infty}u_j\) we showed the existence of $\xi$. 
In fact, if \(n\geq n_0\), then 
\begin{align*}
\sum_{j=1+\Phi(n)}^{\Phi(1+n)}\xi_j p^j=u_{n+1}-u_n=i p^{1+\Phi(n)}. 
\end{align*}
Thus, (\ref{eqn:3-6}) holds with \(\rho(\xi;n+1)=i\). \par
We now prove the uniqueness of $\xi$.  
We see that the sequence \((\xi(j))_{j\geq n_0}\), 
defined by (\ref{eqn:2-4}) for \(x=\xi\), satisfies (\ref{eqn:3-9}) by (\ref{eqn:3-6}) :
$$\xi(j)=\xi(j-1)+\rho(\xi; j) p^{1+\Phi(j-1)},\qquad j\geq 1+n_0,$$
where $\rho(\xi; j)\in\{0\}\cup S(j-1)$.
 Moreover, (\ref{eqn:3-7}) holds by 
\(0\leq \xi(n_0)<p^{1+\Phi(n_0)}\) and (\ref{eqn:3-5}). 
In the same way as the proof of Proposition \ref{prop2}, we can show for any \(j\geq n_0\) that 
\begin{align*}
&f(\xi)-f(\xi(j))\\
&=\sum_{m\geq p^{1+\Phi(j)}} B(m)
\Big(\chi(m;\xi)-\chi(m;\xi(j))\Big)\\
&=\sum_{m\geq p^{1+\Phi(j)}} b(m)p^{\tau(m)}
\Big(\chi(m;\xi)-\chi(m;\xi(j))\Big).
\end{align*}
We fix an integer \(m\) with \(m\geq p^{1+\Phi(j)}(\geq p^{1+\Phi(n_0)})\). Then we have 
\(\chi(m;\xi(j))=0\) by \(\tau(m)\geq j+1>j\geq \tau(\xi(j))\). 
Moreover, \(\chi(m;\xi)=1\) if and only if \(m=\xi(j')\) for some \(j'\geq j+1\). 
In particular, if \(\chi(m;\xi)=1\), then we have \(m\equiv u \mod p^{1+\Phi(n_0)}\). \par
For any integer \(m\) with \(m\geq p^{1+\Phi(j)}\) and \(m\equiv u \mod p^{1+\Phi(n_0)}\), we see by (\ref{yyy}) that 
\[v_p(b(m)p^{\tau(m)})\geq h+\tau(m)\geq 1+h+j.\]
Thus, we get by (\ref{eqn:3-4}) that 
\[f(\xi(j))\equiv f(\xi)=0 \pmod {p^{1+h+j}},\]
which implies (\ref{eqn:3-8}). By the uniqueness of \((u_j)_{j\geq n_0}\), we see 
\(\xi(j)=u_j\) for all \(j\geq n_0\). This concludes the proof of the theorem.
\end{proof}
\begin{exa}\label{exa1}
\begin{rm}
Let \(a=\sum_{i=0}^{\infty} a_i p^i \in\mathbb{Z}_p\) be fixed and \(f:\mathbb{Z}_p\to\mathbb{Z}_p\) be defined by 
\[f(x)=a+\sum_{j=0}^{\infty} x_{2j} p^j,\]
where $x=\sum_{i=0}^{\infty} x_i p^i$.
Let \(\Phi(n):=2n+1\) for any nonnegative integer \(n\). Then we have \(f\in\mathcal{F}(\Phi)\). 
In fact, if \(x,y \in\mathbb{Z}_p\) satisfy 
\(|x-y|_p\leq p^{-1-\Phi(n-1)}=p^{-2n}\) for some positive integer \(n\), then 
\(|f(x)-f(y)|_p\leq p^{-n}\). \par
In what follows, we check that Theorem \ref{thm2} is applicable to \(f\). Putting \(h=0,n_0=0,\) and \(u=-a_0\), we 
see 
\[u<p^{1+\Phi(n_0)}=p^2, \qquad f(u)\equiv 0 {\pmod {p^{1+h+n_0}}}.\]
Moreover, let \(m,n, i\) be nonnegative integers with \(0\leq m<p^{1+\Phi(n)}=p^{2n+2}\) and 
\(1\leq i\leq p-1\). Then we have 
\[f\left(m+i\cdot p^{1+\Phi(n)}\right)-f(m)=i\cdot p^{n+1}.\]
Observe that \(\tau(m+i\cdot p^{1+\Phi(n)})=n+1\) by 
\(p^{1+\Phi(n)}\leq m+i\cdot p^{1+\Phi(n)}<p^{1+\Phi(n+1)}\). Thus the number 
\[b\left(m+i\cdot p^{1+\Phi(n)}\right)(\equiv i \pmod p)\]
satisfies (\ref{eqn:3-3}) with \(S(n)=\{1,2,\ldots,p-1\}\) for any \(n\geq n_0=0\). \par
Similarly, (\ref{eqn:3-3}) with \(S(n)=\{p+1,p+2,\ldots,2p-1\}\) holds for any \(n\geq n_0=0\). 
In fact, we get for any \(1\leq i\leq p-1\) that 
\[f\left(m+(i+p)\cdot p^{1+\Phi(n)}\right)-f(m)=i\cdot p^{n+1}\]
and that \(\tau(m+(i+p)\cdot p^{1+\Phi(n)})=n+1\).
\end{rm}
\end{exa}
\begin{exa}
\begin{rm}
Let \(p=5\) and \(a\in\mathbb{Z}_5\) be fixed, and let \(f:\mathbb{Z}_5\to\mathbb{Z}_5\) be defined by 
\[f(x)=a+\sum_{j=0}^{\infty} x_{2j}^3 5^j.\]
In the same way as in Example \ref{exa1}, 
we have \(f\in\mathcal{F}(\Phi)\), where \(\Phi(n):=2n+1\) for any nonnegative integer \(n\). \par 
Observe that if \(m,n,i\) are nonnegative integers with \(0\leq m<5^{2n+2}\) and \(1\leq i\leq 4\), then we have 
\[f\left(m+i\cdot 5^{1+\Phi(n)}\right)-f(m)=i^3\cdot 5^{n+1}.\]
Moreover, the map \(\phi:\mathbb{Z}/5\mathbb{Z}\to \mathbb{Z}/5\mathbb{Z}\) with 
\[\phi(x \hspace{-1.5mm} \mod 5\mathbb{Z})=x^3 \mod {5\mathbb{Z}}\]
is bijective. 
Thus, in the same way as in Example \ref{exa1}, we can verify that Theorem \ref{thm2} is applicable to \(f\), where 
\(h=0,n_0=0\), \(u\) is a unique integer with 
\(0\leq u\leq 4\) and \(a+u^3 \equiv 0 \pmod 5\), and 
\(S(n)=\{1,2,3,4\}\) for any \(n\geq n_0=0\). 
\end{rm}
\end{exa}
We give an example of Theorem \ref{thm2} satisfying \(h\ne 0\) in Section 4.

\section{Approximability of continuous functions}\label{sec4}
In this section we use notation defined in Sections 2 and 3. Let \(f:\mathbb{Z}_p\to\mathbb{Z}_p\) be a function 
and \(u\in\mathbb{Z}_p\). Recall for a positive integer \(s\) that \(f\) is differentiable modulo \(p^s\) 
at \(u\) if there exists a \(p\)-adic number \(\partial_s f(u)\) satisfying the following: 
If \(n\) is a sufficiently large integer, then for any \(u'\in\mathbb{Z}_p\), we have 
\begin{align}
\label{modmod}
f(u+p^n u')\equiv f(u)+p^n u' \partial_s f(u) \pmod {p^{n+s}}.
\end{align}
Axelsson and Khrennikov showed for 1-Lipschitz functions \(f\) that if 
\(|\partial_s f(u)|_p=1\), then the Hensel's lifting process can be used (Theorem 2.4 and Corollary 2.6 in \cite{axe}). \par
For the application of Theorem \ref{thm2}, we now modify the notion of differentiability modulo \(p^s\) to consider the case where \(f\) is a general continuous 
function, using a strictly increasing function \(\Phi:\mathbb{N}\to\mathbb{N}\). Let \(u\in\mathbb{Z}_p\). 
We call \(f\) approximable at \(u\) with respect to \(\Phi\) if there exist nonnegative integers \(h=h(u)\), \(l=l(u),\) 
and a sequence \((\delta_n f) (u)\) (\(n\geq l\)) of \(p\)-adic numbers satisfying the following: 
For any integer \(n\geq l\) and \(u'\in\mathbb{Z}_p\), we have 
\begin{align}
f\left(u+p^{1+\Phi(n-1)}u'\right)\equiv f(u)+p^{h+n} u'\cdot (\delta_n f) (u) \pmod {p^{h+n+1}}.
\label{eqn:4-1}
\end{align}
The approximability of continuous functions generalizes the differentiability modulo \(p^s\). 
Thus, Corollary \ref{cor} is also applicable to differentiable function modulo \(p^s\). 
In fact, let $\Phi=Id$ and let \(f\) be a continuous function which is differentiable modulo \(p^s\) at \(u\). 
Let \(h=s-1\). 
Then (\ref{modmod}) implies that $$f(u+p^n u')=f\left(u+p^{1+\Phi(n-1)}u'\right)\equiv 
f(u)+p^{h+n} u'\cdot (\delta_n f) (u) \pmod {p^{h+n+1}},$$
where 
$(\delta_n f) (u)=\partial_{h+1} f (u)/p^h$. \par
In order to apply Theorem \ref{thm2} we will make use of the fact that $(\delta_n f) (u)$ is a unit in $\mathbb{Z}_p$. 
Let \(\Lambda\) be a subset of \(\mathbb{Z}_p\). 
We call \(f\) uniformly 
approximable on \(\Lambda\) with respect to \(\Phi\) if 
\(f\) is approximable at any \(u\in \Lambda\) with respect to \(\Phi\) and if 
\(h(u),l(u)\) are constants on \(\Lambda\) which we call \(h(\Lambda)\) and \(l(\Lambda)\) in that case. 
\begin{cor}
Let \(\Phi:\mathbb{N}\to\mathbb{N}\) be a strictly increasing function and \(f\in \mathcal{F}(\Phi)\). 
Let \(n_0,u\in\mathbb{N}\) with \(u<p^{1+\Phi(n_0)}\). Assume that \(f\) is uniformly approximable on the set 
\[\Lambda:=\left\{u'\in\mathbb{Z}_p\left| u'\equiv u\pmod {p^{1+\Phi(n_0)}}\right.\right\}\]
with respect to \(\Phi\). Moreover, assume that 
\begin{align*}
&n_0+1\geq l(\Lambda),\\
&f(u')\equiv 0 \pmod {p^{h(\Lambda)+n_0+1}}
\end{align*}
and that 
\[|(\delta_n f) (u')|_p=1\]
for any \(n\geq n_0\) and \(u'\in \Lambda\). 
Then there exists a unique \(\xi\in\mathbb{Z}_p\) such that 
\begin{align*}
&f(\xi)=0,\\
&\xi\equiv u\pmod {p^{1+\Phi(n_0)}},
\end{align*}
and such that, for any \(n\geq n_0\), 
\begin{align*}
\rho(\xi;n+1)\in \{0,1,\ldots,p-1\}.
\end{align*}
\label{cor}
\end{cor}
\begin{proof}
We apply Theorem \ref{thm2} with \(h=h(\Lambda)\) and \(S(n):=\{1,\ldots,p-1\}\) for all 
\(n\geq n_0\). 
Since \(f(u)\equiv 0 \pmod {p^{h(\Lambda)+n_0+1}}\) by the assumption of Corollary \ref{cor}, 
it suffices to check (\ref{eqn:3-3}) and (\ref{yyy}). Assume that \(n,m,i\in\mathbb{N}\) satisfy \(n\geq n_0\), 
\(m<p^{1+\Phi(n)}\), \(m\equiv u \pmod {p^{1+\Phi(n_0)}}\), and \(1\leq i\leq p-1\). Since 
\(\tau(m+ip^{1+\Phi(n)})=n+1\), we see by (\ref{eqn:4-1}) and by $n+1\geq n_0+1\geq l(\Lambda)$ that 
\begin{align*}
b(m+i p^{1+\Phi(n)})&=p^{-n-1}(f(m+i p^{1+\Phi(n)})-f(m))\\
&\equiv p^{h(\Lambda)} i\cdot  (\delta_{n+1} f) (m) \pmod {p^{h(\Lambda)+1}},
\end{align*}
which implies (\ref{eqn:3-3}) by \(|(\delta_n f) (m)|_p=1\). \par
Now, let \(m\) be an integer with \(m\geq p^{1+\Phi(n_0)}\), \(m\equiv u \pmod {p^{1+\Phi(n_0)}}\). 
Put \(u':=m-M(m)\) and \(n:=\tau(m)-1\geq n_0\). Recall that
$M(m)=\sum_{i=\Phi(n)+1}^{\Phi(n+1)} m_i p^i$ and therefore
\(u'':=p^{-1-\Phi(n)}M(m)\) is an integer. 
Thus, we see by \(u'\in\Lambda\), (\ref{eqn:4-1}) and $n+1\geq n_0+1\geq l(\Lambda)$ that 
\[f(u'+p^{1+\Phi(n)} u'') \equiv f(u')+p^{h(\Lambda)+n+1} u''\cdot (\delta_{n+1} f) (u') \pmod {p^{h(\Lambda)+n+2}}.\]
Finally, since \(m=u'+p^{1+\Phi(n)} u''\) we obtain 
\begin{align*}
b(m)&=p^{-n-1}(f(u'+p^{1+\Phi(n)} u'')-f(u'))\\
&\equiv p^{h(\Lambda)} u''\cdot  (\delta_{n+1} f) (u') \equiv 0 \pmod {p^{h(\Lambda)}}, 
\end{align*}
which implies (\ref{yyy}). 
\end{proof}
\begin{exa}
\begin{rm}
Let \(a\in\mathbb{Z}_p\) be fixed. If \(p\ne 2\), then assume that \(a\equiv 1 \pmod p\). If \(p=2\), then suppose that 
\(a\equiv 1\pmod 8\). For any \(x\in \mathbb{Z}_p\), put 
\[f(x):=-a+\left(\sum_{j=0}^{\infty} x_{2j} p^j\right)^2.\]
Then we have \(f\in\mathcal{F}(\Phi)\), where \(\Phi(n)=2n+1\) for any nonnegative integer \(n\). 
Put 
\begin{align*}
n_0=h=\begin{cases}
0 & \mbox{if }p\ne 2, \\
1 & \mbox{if }p=2.
\end{cases}
\end{align*}
and 
\begin{align*}
l=\begin{cases}
1 & \mbox{if }p\ne 2, \\
2 & \mbox{if }p=2.
\end{cases}
\end{align*}
Note that in the case of $p=2$ the value of $h$ is different from 0. \par 
First, there exists \(u\in\mathbb{Z}_p\) with 
\[f(u)\equiv 0 \pmod {p^{h+n_0+1}}.\]
Let \(\Lambda:=\{u'\in\mathbb{Z}_p\mid u'\equiv u\pmod {p^{1+\Phi(n_0)}}\}\). We show that \(f\) 
is uniformly approximable on \(\Lambda\) with respect to \(\Phi\) with 
\(h(\Lambda)=h\), \(l(\Lambda)=l\) and 
\begin{align}
(\delta_n f)(x)=\begin{cases}
2 & \mbox{if }p\ne 2, \\
1 & \mbox{if }p=2
\end{cases}
\label{eqn:4-4}
\end{align}
for any \(x\in \Lambda\). Thus, Corollary \ref{cor} is applicable to \(f\). \par
Let \(x\in \Lambda\). Let \(n\) be an integer with \(n\geq l\) and \(y\in\mathbb{Z}_p\). Since 
\[
x+p^{1+\Phi(n)}y=\sum_{j=0}^{2n+1} x_j p^j
+
\sum_{j=2n+2}^{\infty} (x_j+y_{j-2n-2}) p^j,
\]
we get after a direct calculation
\begin{align}
&f\left(x+p^{1+\Phi(n)}y\right)-f(x)\nonumber\\
&=\left(
\sum_{j=0}^{n} x_{2j} p^j
+
\sum_{j=n+1}^{\infty} (x_{2j}+y_{2j-2n-2}) p^j
\right)^2
-
\left(
\sum_{j=0}^{\infty} x_{2j} p^j
\right)^2\nonumber\\
& \equiv
2\sum_{j=0}^{n} x_{2j} p^j\cdot 
\sum_{j=n+1}^{\infty} y_{2j-2n-2} p^j
\pmod {p^{2n+2}}.
\label{eqn:4-5}
\end{align}
Note that \(x_0\equiv 1\pmod p\) by \(x\in \Lambda\). If \(p\) is odd, then (\ref{eqn:4-5}) implies that 
\begin{align}
f\left(x+p^{1+\Phi(n)}y\right)-f(x)
\equiv 2x_0 y_0 p^{n+1}\equiv 2y p^{n+1}\pmod {p^{n+2}}.
\label{eqn:4-6}
\end{align}
In the case of \(p=2\), we get by \(n\geq 1\) that 
\begin{align}
f\left(x+2^{1+\Phi(n)}y\right)-f(x)
\equiv 2x_0 y_0 \cdot 2^{n+1}\equiv y 2^{n+2}\pmod {2^{n+3}}.
\label{eqn:4-7}
\end{align}
Thus, (\ref{eqn:4-6}) and (\ref{eqn:4-7}) imply the approximability and (\ref{eqn:4-4}).
\end{rm}
\end{exa}

\section*{Acknowledgements}

The first author is supported by JSPS KAKENHI Grant Number 15K17505. The second author acknowledges the support of the bilateral project ANR-FWF (France-Austria) called MUDERA (Multiplicativity, Determinism, and Randomness), ANR-14-CE34-0009.


\begin{thebibliography}{99}
\bibitem{AKbook}
V. Anashin, A. Khrennikov, 
Applied Algebraic Dynamics, de Gruyter Exp. Math., vol. 49, Walter de Gruyter, Berlin, New York, 2009.
\bibitem{ana}
V. Anashin, A. Khrennikov, E. Yurova, 
\(T\)-functions revisited: new criteria for bijectivity/transitivity, \textit{Des. Codes Cryptogr.}
{\bf 71} (2014), 383--407. 
\bibitem{axe}
E. Y. Axelsson, A. Khrennikov, 
Generalization of Hensel's lemma: Finding the roots of \(p\)-adic Lipschitz functions, 
\textit{J. Number Theory} {\bf 158} (2016), 217--233.
\bibitem{Sch}
W.H. Schikhof, 
Ultrametric Calculus. An Introduction to $p$-adic Analysis, Cambridge University Press, Cambridge, 1984.
\bibitem{vdP1}
M. van der Put, 
Alg\`ebres de fonctions continues $p$-adiques. I, Nederl. Akad. Wetensch. Proc. Ser. A {\bf 72} (in French), \textit{Indag. Math.} {\bf 30} (1968) 401--411. 
\bibitem{vdP2} 
M. van der Put,
Alg\`ebres de fonctions continues $p$-adiques. II, Nederl. Akad. Wetensch. Proc. Ser. A {\bf 71} (in French), \textit{Indag. Math.} {\bf 30} (1968) 412--420. 
\end{thebibliography}
\end{document}